\documentclass{article}
\usepackage{amsmath, amssymb, amsthm,mathrsfs,mathtools,comment}
\usepackage{geometry}
\title{Intermediate subalgebras of cross-product von Neumann algebras arising from weakly compact actions}
\author{Ionu\c{t} Chifan and Zachary Flood}
\date{}

\numberwithin{equation}{section}

\usepackage{xcolor}

\usepackage{hyperref}
\hypersetup{linktocpage}

\hypersetup{colorlinks,
    linkcolor={red!50!black},
    citecolor={blue!80!black},
    urlcolor={blue!80!black}}
\usepackage{float}

\newtheorem{theorem}{Theorem}[section]
\newtheorem{introthm}{Theorem}

\newtheorem{introcor}[introthm]{Corollary}
\newtheorem{lemma}[theorem]{Lemma}
\newtheorem{corollary}[theorem]{Corollary}
\newtheorem{claim}[theorem]{Claim}
\newtheorem{prop}[theorem]{Proposition}

\theoremstyle{definition}
\newtheorem{definition}[theorem]{Definition}

\theoremstyle{remark}

\newcommand{\M}{\mathcal M}
\newcommand{\A}{\mathcal A}
\newcommand{\B}{\mathcal B}
\newcommand{\Q}{\mathcal Q}
\newcommand{\eL}{\mathcal L}
\newcommand{\Z}{\mathbb Z}
\newcommand{\C}{\mathbb C}
\newcommand{\tr}{\operatorname{tr}}
\newcommand{\F}{\mathbb F}

\begin{document}

\maketitle
\begin{abstract} \noindent In this short paper, we establish the following Galois-type property for all measure preserving, weakly compact actions $G \curvearrowright (X,\mu) $ of a non-inner amenable group $G$ on a standard probability space $X$: for any intermediate von Neumann subalgebra  $\mathcal L(G)\subseteq \mathcal N \subseteq L^\infty (X)\rtimes G$, there is a factor $G\curvearrowright Y$ of $G \curvearrowright X$ such that $\mathcal N= L^\infty(Y)\rtimes G$. In particular, this holds for all pmp distal actions $G\curvearrowright X$, thereby answering positively an open problem from \cite{BCCMSW24} (which can also be traced back to \cite{CD20}). We also provide an example showing our result is optimal.
\end{abstract}

\section*{Introduction}

A natural Galois-type problem for group measure space von Neumann algebras is to understand to what extent intermediate von Neumann algebras reflect factors of the underlying action. More precisely, let $G$ be a countable icc group, let $G\curvearrowright (X,\mu)$ be a free ergodic pmp action, and set $\mathcal M=L^\infty(X)\rtimes G$. Every factor \(G\curvearrowright (Y,\nu)\) of \(G\curvearrowright (X,\mu)\) gives rise to an intermediate von Neumann algebra $$\mathcal L(G)\subseteq L^\infty(Y)\rtimes G\subseteq L^\infty (X)\rtimes G.$$ 

\noindent It is therefore natural to investigate a converse to this; namely, under what assumptions on $G$ and the action $G\curvearrowright X$ every intermediate von Neumann algebra

\begin{equation*}
\mathcal L(G)\subseteq \mathcal N\subseteq L^\infty(X)\rtimes G
\end{equation*}
is of the form $\mathcal N=L^\infty(Y)\rtimes G$ for some factor \(G\curvearrowright Y\) of \(G\curvearrowright X\). 
\vskip 0.05in
\noindent The icc assumption is natural in this problem, as it guarantees that $\mathcal L(G)$ is a factor. Moreover, without such an assumption the structure of intermediate algebras can be substantially different. For instance, if $G$ is abelian and $G\curvearrowright X$ is a free ergodic compact action, then $\mathcal L(G)$ is itself a Cartan subalgebra of $\M$ \cite{Pac85}, and intermediate algebras are governed by the much larger lattice of subequivalence relations rather than merely by factors of the original action. There are also more direct obstructions arising from the failure of the icc condition: the presence of nontrivial central elements, or more generally nontrivial central projections in $\mathcal L(G)$, can be used to construct von Neumann subalgebras which do not have the expected dynamical origin; see, for instance, the related construction in \cite[Proposition~3.1]{AJ23}.
\vskip 0.05in
\noindent This problem is closely related to the Neshveyev--St{\o}rmer rigidity phenomenon and to Popa's deformation/rigidity results for Bernoulli actions \cite{Po03,Po04}, although those results concern a different rigidity question. For the intermediate algebra problem above, Chifan--Das proved that the answer is positive for compact ergodic actions of arbitrary i.c.c.\ groups, as part of their more general description of intermediate crossed products associated with compact extensions \cite{CD20}. Independently, Jiang--Skalski obtained the result for profinite actions and observed that their argument extends to arbitrary compact actions \cite{JS21}. Chifan--Das also showed, for an arbitrary free ergodic pmp action, that the conclusion holds for intermediate algebras contained in the quasinormalizing algebra of $\mathcal L(G)$; in particular, all finite-index intermediate subfactors arise from finite transitive factors of the action. More recently, further positive results have been obtained for certain rigid actions of convergence groups and higher-rank lattices \cite{AGG24}. Nevertheless, the general problem remains widely open. It was explicitly formulated in \cite{BCCMSW24}, where it was pointed out that very little is known beyond the compact case; notably, it is not known in general even for distal actions of length two, that is, for a compact extension of a compact action see the comments after \cite[Open Problem, page 48]{BCCMSW24}.

\noindent In this short paper we make new progress on this problem by identifying a new class of groups/actions for which the intermediate property holds.

\begin{introthm}\label{mainthm:1} Let $G$ be a non-inner amenable group and let $G\curvearrowright^\sigma (\A, \tau)$ be a weakly compact trace-preserving action on an abelian von Neumann algebra $\A$. Then for any intermediate von Neumann subalgebra $\eL(G)\subseteq \Q \subseteq \A\rtimes G$ there is a $G$-invariant von Neumann subalgebra $\B\subseteq \A$ such that $\Q= \B\rtimes G$. 
\end{introthm}

\noindent The notion of weakly compact action was introduced in \cite{OP07} as a natural generalization of the classical notions of profiniteness and compactness (or discrete spectrum). It played an important role in the study of uniqueness of Cartan subalgebras and has since had a significant impact on the classification of cross-product von Neumann algebras. For its precise definition see Definition \ref{def:weak compact}. In particular, the class of weakly compact actions contains all profinite and compact actions and, more generally, all distal actions. The latter fact was first established by J.\ Peterson and the first author in \cite{CP11}. For the reader's convenience, we include some of these arguments in this paper as well, see Section \ref{sec: weak compactness}.    

\noindent  In particular, Theorem \ref{mainthm:1}  solves affirmatively the aforementioned question which was left open in \cite{BCCMSW24}. In fact, more can be shown.  

\begin{introcor}\label{main: cor1} For any non-inner amenable group $G$ and any measure preserving free ergodic distal action on a standard probability space $G\curvearrowright^\sigma X$ its crossed product von Neumann algebra $L^\infty(X)\rtimes G$ satisfies the intermediate von Neumann algebra property.  
    
\end{introcor}

\noindent We observe that our result is in fact  optimal in  natural sense. Specifically, if one consider actions that are not weakly compact the result does not hold true anymore. Indeed, building upon some ideas and constructions introduced by Amrutam-Glasner-Glasner in their recent work \cite[Propositions 5.2 and 5.4, Theorems 5.9 and 5.14]{AGG24b} we can show the following. 

\begin{introthm}\label{prop:main}
Consider the groups  $H=\Z_3$, $ \Gamma=\F_2$, $K=H^{(\Gamma)}$, and $G=K\rtimes\Gamma=H\wr\Gamma$. Let $X:=H^\Gamma$ with product uniform measure and consider the following action $G \curvearrowright X$. For every $ag\in G$ ($a\in K$, $g\in \Gamma$), $h\in \Gamma$ and $(x_h)_h\in X$ we let  
\begin{equation*}\label{eq:action}
 ((ag)\cdot x)_h=a_h+x_{g^{-1}h}.
\end{equation*}
Then $G$ is icc, non-inner amenable and $G \curvearrowright X$ is a pmp, free, weak mixing action.
There exists an intermediate $\mathrm{II}_1$ factor $\mathcal N$ such that $\mathcal L (G)\subsetneq \mathcal N\subsetneq L^\infty(X)\rtimes G$ and $\mathcal N\cap L^\infty(X)=\C1$.
In particular, there is no $G$-invariant von Neumann subalgebra
$\mathcal C\subset L^\infty(X)$ such that $\mathcal N=C\rtimes G$ inside the given
crossed product.
\end{introthm}

\noindent In particular, this example provides a negative answer to
 \cite[Open Problem page 48]{BCCMSW24}. It also suggests that the
intermediate subalgebra property depends on a delicate interplay
between the acting group and the action. For weakly mixing actions,
non-inner amenability of the acting group is not sufficient, and it
is natural to ask whether additional assumptions, such as triviality
of the amenable radical, can lead to positive results. Determining
which combinations of group-theoretic and dynamical assumptions
ensure this property appears to be a subtle problem and an
interesting direction for future work.

\noindent\textbf{Acknowledgements.} The first author was supported in part by the NSF grant DMS-2452247. The second author was supported in part by the NSF grant DMS-2452247 and the Erwin and Peggy Kleinfeld Graduate Fellowship. 

\noindent \textbf{AI statement.}  ChatGPT Astra 6 was used to identify several relevant references and to help constructing the counterexample presented in Theorem C. This tool was also used to improve the language and organization of some sections of the paper.    


\section{Spectral gap}

\noindent In this section we recall several spectral gap inequalities that will be very useful to deriving our main results. 

\begin{theorem}\label{specgap1} Let $G$ be a non-inner amenable group and let $G\curvearrowright^{\sigma} (\Q, \tau)$ be a trace preserving action of a tracial von Neumann algebra. Denote by $\M =\Q\rtimes G$ the corresponding crossed product von Neumann algebra. Then there exist  a finite set $K\subset G$ and $C>0$ such that for all $\xi\in L^2(\M)$ we have 

\begin{equation}\label{spectralgapreltoQ}
    \sum_{g\in K}\| u_g \xi -\xi u_g \|^2_2\geq C\|\xi- \mathbb E_{\Q}(\xi)\|^2_2.
\end{equation}  
\end{theorem}

\begin{proof} Since $G$ is non-inner amenable there exist a finite subset $F \subset G$ and a scalar $C>0$ such that for all $\eta\in \ell^2(G)$ we have 

\begin{equation}\label{inneramenab}
    \sum_{g\in F}\|u_g \eta -\eta u_g\|^2_2\geq C\|\eta-\tau(\eta)\|_2^2.
\end{equation}    

\noindent To show \eqref{spectralgapreltoQ}, it suffices to show there exists a finite subset $K\subset G$ and $C>0$ such that for every $\eta\in L^2(\M)\ominus L^2(\Q)=:\mathcal{H}$ we have $\sum_{g\in K}\|u_g\eta u_g^*-\eta\|_2^2\geq C\|\eta\|_2^2$. Indeed, for $\xi\in L^2(\M)$ and $\eta:=\xi-\mathbb E_{\Q}(\xi)$, we have $\mathbb P_{\mathcal H}(u_g\xi u_g^*-\xi)=(u_g\xi u_g^*-\xi)-\mathbb E_{\Q}(u_g\xi u_g^*-\xi)=u_g\eta u_g^*-\eta$. 
\vskip 0.03in
\noindent Now, we assume by contradiction that there exists a sequence $(\eta_n)_n\subset \mathcal{H}$ with $\|\eta_n\|_2=1$ for every $n$ and for which $\lim_n\|u_g\eta_nu_g^*-\eta_n\|_2=0$ for every $g\in G$. Write $\eta_n=\sum_{s\in G}\eta_{n,s}\otimes \delta_s$, for $\eta_{n,s}\in L^2(\Q)$ and $\eta_{n,1}=0$. Define $f_n(s)=\|\eta_{n,s}\|_2$ and notice that $f_n\in\ell^2(G)$ with $\|f_n\|_2=\|\eta_n\|_2$. Moreover,

\begin{align*}
    \|f_n-f_n\circ\text{Ad}(g^{-1})\|_2^2&=\sum_{g\in G}\left|\|\eta_{n,s}\|_2-\|\eta_{n,g^{-1}sg}\|_2\right|^2\\
    &\leq \sum_{s\in G}\|\eta_{n,s}-\sigma_g(\eta_{n,g^{-1}sg})\|_2^2=\left\|\sum_{s\in G}(\eta_{n,s}-\sigma_g(\eta_{n,g^{-1}sg}))\otimes\delta_s\right\|_2^2\\
    &=\left\|\sum_{s\in G}\eta_{n,s}\otimes\delta_s-\sum_{s\in G}\sigma_g(\eta_{n,s})\otimes\delta_{gsg^{-1}}\right\|_2^2\\
    &=\|\eta_n-u_g\eta_nu_g^*\|_2^2\to 0.
\end{align*}

\noindent From \eqref{inneramenab}, this implies $\lim_n\|f_n-f_n(e)\|_2=0$. Since $f_n(e)=0$, it follows that $\lim_n\|\eta_n\|_2=\lim_n\|f_n\|_2=0$, a contradiction to the fact $\|\eta_n\|_2=1$ for all $n$.
\end{proof}

\vspace{1mm}

\begin{corollary}\label{corspectralgap} Let $G$ be  a non-inner amenable group and let $G\curvearrowright (\A, \tau)$ be a trace preserving action of a tracial von Neumann algebra.  Denote by $\M =\A\rtimes G$ the corresponding crossed product von Neumann algebra. Then there exist  a finite set $K\in G$ and $S>0$ such that for all $\xi\in L^2(\M\bar \otimes \M)$ we have \begin{equation}
        \sum_{g\in K}\| (u_g \otimes u_g)\xi -\xi (u_g \otimes u_g) \|^2_2\geq S\|\xi- \mathbb E_{\A \bar\otimes \A }(\xi)\|^2_2.
    \end{equation}      
\end{corollary}

\begin{proof} Notice that we have the canonical semidirect product decomposition $G\times G = (G\times \{1\}) \rtimes d(G)$, where $d(G)=\{(g,g) \,|\, g\in G \}$ and the action of $d(G)$ on $G\times \{1\}$ is given by conjugation. Therefore, we have the cross-product decomposition $\M\bar\otimes \M = ((\A \bar \otimes \A)\rtimes (G\times \{1\}))\rtimes d(G)= (\M \bar \otimes \A)\rtimes d(G)$.
\vskip 0.05in
\noindent By Theorem \ref{specgap1} there exists a finite subset $E \subset d(G)$ and a scalar $C>0$ such that for all $\xi\in L^2(\M\bar{\otimes}\M)$, we have

\begin{equation}\label{specgap2}
    \sum_{g\in E}\| (u_g \otimes u_g)\xi -\xi (u_g \otimes u_g) \|^2_2\geq C\|\xi- \mathbb E_{\M \bar\otimes \A }(\xi)\|^2_2.
\end{equation}

\noindent Similarly, using the decomposition  $G\times G= (\{1\}\times G)\rtimes d(G)$, there exists a finite subset $F\subset d(G)$ and a scalar $D>0$ such that for all for $\eta\in L^2(\M\bar{\otimes}\M)$ we have

\begin{equation}\label{specgap3}
    \sum_{g\in F}\| (u_g \otimes u_g)\eta -\eta(u_g \otimes u_g) \|^2_2\geq D\|\eta- \mathbb E_{\A \bar\otimes \M }(\eta)\|^2_2.
\end{equation}

\noindent Denote by $K:=E\cup F$ and $S:=\min\{C,D\}>0$. Now fix $\xi\in L^2(\M\bar{\otimes}\M)$, and let $\eta:=\mathbb E_{\M\bar{\otimes}\A}(\xi)$. We will apply \eqref{specgap2} for $\xi$ and \eqref{specgap3} for $\eta$. First, notice that $\mathbb E_{\A\bar{\otimes}\M}(\eta)=\mathbb E_{\A\bar{\otimes}\M}(\mathbb E_{\M\bar{\otimes}\A}(\xi))=\mathbb E_{\A\bar{\otimes}\A}(\xi)$, and that $\|(u_g\otimes u_g)\xi-\xi(u_g\otimes u_g)\|_2\geq \|\mathbb E_{\M\bar{\otimes}\A}((u_g\otimes u_g)\xi-\xi(u_g\otimes u_g))\|_2=\|(u_g\otimes u_g)\eta-\eta(u_g\otimes u_g)\|_2$.
\vskip 0.05in
\noindent Thus, adding inequalities \eqref{specgap2} and \eqref{specgap3} we get that 

\begin{align*}
    \sum_{g\in K}\| (u_g \otimes u_g)\xi -\xi (u_g \otimes u_g) \|^2_2&\geq \sum_{g\in E}\| (u_g \otimes u_g)\xi -\xi (u_g \otimes u_g) \|^2_2+\sum_{g\in F}\| (u_g \otimes u_g)\eta -\eta (u_g \otimes u_g) \|^2_2\\
    &\geq S(\|\xi- \mathbb E_{\M \bar\otimes \A }(\xi)\|^2_2+\|\mathbb E_{\M\bar{\otimes}\A}(\xi)-\mathbb E_{\A\bar{\otimes}\A}(\xi)\|_2^2)\\
    &=S\|\xi-\mathbb E_{\A\bar{\otimes}\A}(\xi)\|_2^2.
\end{align*}
\end{proof}

\section{Weakly compact actions}\label{sec: weak compactness}

We recall some notions from ergodic theory.

\begin{definition}
    Let $(X,\mathcal{D},\mu)$ be a standard probability space and let $G$ be a countable group.  
    A \emph{probability measure-preserving (abbrev.\ pmp)} action of $G$ on $(X,\mathcal{D},\mu)$ is a homomorphism
     $G \to \mathrm{Aut}(X,\mathcal{D},\mu)$,
    given by  $\gamma \mapsto T_\gamma$,
  where each $T_\gamma$ is a measurable bijection preserving $\mu$. This will be denoted by $G \curvearrowright (X, \mathcal D, \mu)$ or, very often, when the $\sigma$-algebra $\mathcal D$ is self-understood, simply $G \curvearrowright(X, \mu)$.
\end{definition}

\begin{definition}
    Let $G \curvearrowright (X,\mathcal{D},\mu)$ and $G \curvearrowright (Y,\mathcal{C},\nu)$ be pmp actions.  
    A \emph{factor map} is a measurable, $G$-equivariant map $\pi:X\to Y$ such that $\pi_*\mu=\nu$.  
    We say that $G \curvearrowright (X,\mu)$ is an \emph{extension} of $G \curvearrowright(Y,\nu)$ and write $G \curvearrowright (X,\mu) \to G \curvearrowright (Y,\nu)$.
    \end{definition}
\noindent To introduce our results we also recall the notion of compact extensions. 
\begin{definition}\label{compextensionpmp}
    Let $G \curvearrowright (X,\mu) \to G \curvearrowright(Y,\nu)$ be an extension. We say that it is a \emph{compact extension} if $L^2(X,\mu)$ is the closed linear span of its finitely generated, $G$-invariant $L^\infty(Y,\nu)$-submodules. Here a closed $L^\infty(Y,\nu)$-submodule is finitely generated if it is the $\|\cdot\|_2$-closed linear span of $L^\infty(Y,\nu)\xi_1,\dots,L^\infty(Y,\nu)\xi_k$ for some $\xi_1,\dots,\xi_k\in L^2(X,\mu)$.
\end{definition}

\noindent With these preparations at hand, we now recall the notion of distal tower from \cite{Zi76,Fu77}.

\begin{definition}
    Let $G \curvearrowright (X,\mathcal{D},\mu)$ be a pmp action. A \emph{distal tower} (or \emph{compact tower}) is a transfinite sequence of factors
    \begin{equation*}
    G \curvearrowright(X_0,\mu_0) \leftarrow G \curvearrowright (X_1,\mu_1) \leftarrow \cdots
    \leftarrow G \curvearrowright (X_\alpha,\mu_\alpha) \leftarrow \cdots
    \end{equation*}
    indexed by ordinals, satisfying the following conditions:

\begin{enumerate}
    \item $X_0$ is the trivial one-point action.
    \item For every successor ordinal $\alpha+1$, the extension $G \curvearrowright(X_{\alpha+1},\mu_{\alpha+1}) \to G \curvearrowright (X_\alpha,\mu_\alpha)$ is compact.
    \item For every limit ordinal $\lambda$, the system $G \curvearrowright (X_\lambda,\mu_\lambda)$ is the inverse limit of
    the systems $G \curvearrowright(X_\alpha,\mu_\alpha)$ for $\alpha<\lambda$.
    \item There exists an ordinal $\eta$ such that $G \curvearrowright (X_\eta,\mu_\eta)$ is isomorphic to $G \curvearrowright (X,\mu)$.
\end{enumerate}
\end{definition}

\begin{definition}
    A pmp action $G \curvearrowright(X,\mathcal{D},\mu)$ is called \emph{distal} if it admits a distal tower as above; equivalently, it is obtained from the trivial system by iteratively taking compact extensions
    (and inverse limits at limit stages).
\end{definition}

\noindent For an arbitrary ergodic pmp action $G\curvearrowright(X,\mu)$, we also consider its \emph{canonical Furstenberg--Zimmer tower}. Starting with the trivial factor, at every successor stage we take the maximal compact intermediate extension inside $X$, and at every limit stage we take the inverse limit. Since $(X,\mu)$ is a standard probability space, this tower stabilizes at a countable ordinal $\eta$. Its terminal factor $X_\eta$ is the maximal distal factor of $X$, and the extension $X\to X_\eta$ is relatively weakly mixing. The action is distal precisely when $X_\eta=X$, modulo null sets; see \cite{Zi76,Fu77}.
\vskip 0.05in

\noindent To establish our results in this section, we shall translate many of the notions above into the language of von Neumann algebras.

\begin{definition}
    Let $G$ be a countable, discrete group and $(\B,\tau)$ a tracial von Neumann algebra. An \emph{action} $\sigma$ of $G$ on $(\B,\tau)$, denoted by $G \curvearrowright^\sigma\mathcal B$  is a homomorphism from $G$ into $\text{Aut}(\B)$. We say an action $\sigma$ is \emph{trace-preserving} if its range lies in $\text{Aut}(\B,\tau)$, the group of trace-preserving automorphisms of $\B$, i.e.\ $\tau(\sigma_g(x))=\tau(x)$ for all $x\in \mathcal B$, $g\in G$.
\end{definition}

\noindent In the setting of abelian von Neumann algebras with separable predual, any pmp action $G\curvearrowright^\sigma (X,\mu)$ of a countable discrete group $G$ on a standard probability space $(X,\mu)$ yields a pullback action of $G$ on the abelian von Neumann algebra $L^\infty(X,\mu)$ given by $(gf)(x)=f(g^{-1}x)$ for all $x\in X$, $g\in G$, $f\in L^\infty(X)$. The canonical trace on $L^\infty(X,\mu)$ given by $f\mapsto \int_Xf\ d\mu$ is clearly preserved by this action under the assumption it is pmp. Conversely, any trace-preserving action of $G$ on an abelian von Neumann algebra $\A=L^\infty(X,\mu)$ with separable predual sends indicator functions of $\mu$-measurable sets to indicator functions of sets of equal $\mu$-measure, thereby yielding a pmp action of $G$ on the underlying probability space $(X,\mu)$.

\begin{definition}
    A \emph{(trace-preserving) extension of von Neumann algebras} $G\curvearrowright(\A\subset \B)$ is an inclusion of von Neumann algebras $\A\subset \B$ and a (trace-preserving) action $G\curvearrowright \B$ that leaves the subalgebra $\A$ $G$-invariant.
\end{definition}
\noindent Next we recall the notion of compact extension of tracial von Neumann algebras from \cite[Definition 3.9]{CD20}.
 \begin{definition}\label{compactext}
    We say an extension action of tracial von Neumann algebras $G\curvearrowright^\sigma(\A\subset \B)$ is \emph{compact} if there exists $\mathcal{F}\subset \B$ satisfying:
    \begin{enumerate}
        \item $\overline{\text{span}\{\mathcal{F}\}}^{\|\cdot\|_2}=L^2(\B)$.
        \item For every $f\in \mathcal{F}$ and $\varepsilon>0$, there exist $\xi_1,\dots,\xi_n\in L^2(\B)$ so that for every $g\in G$, there exist $\kappa_i(g)\in \A$ for all $1\leq i\leq n$ so that $\sup_{1\leq i\leq n,g\in G}\|\kappa_i(g)\|_\infty<\infty$ and
        \begin{equation*}
            \left\lVert\sigma_g(f)-\sum_{i=1}^n\kappa_i(g)\xi_i\right\rVert_2<\varepsilon
        \end{equation*}
    \end{enumerate}
\end{definition}

\noindent Given a pmp action $G\curvearrowright (X,\mu)$, it can be shown that $G \curvearrowright(X,\mu)\to G \curvearrowright (Y,\nu)$ is a (compact) extension as in Definition \ref{compextensionpmp} if and only if $G\curvearrowright(L^\infty(Y,\nu)\subset L^\infty(X,\mu))$ is a (compact) trace-preserving extension as in Definition \ref{compactext}.
\vskip 0.05in
\noindent In addition to compactness, we consider a strictly weaker property of extensions called \emph{weak compactness}. Unlike compactness, this property is transitive for pairs of extensions and thus it is more robust for applications.

\begin{definition}\label{def:weak compact} (\cite[Definition 3.1]{OP07}, in its equivalent $L^1$-formulation) A trace preserving action $G\curvearrowright^\sigma (\A, \tau)$ on an abelian von Neumann algebra $\A$ is called weakly compact if there exists a net of positive elements $(\xi_n)_n \subset L^1(\A\bar \otimes \A)$ of norm one such that 

\begin{enumerate}
    \item $
    \lim_n\| \sigma_{g}\otimes \sigma_g(\xi_n)-\xi_n\|_1=0$ for all $g\in G$;
    \item $\lim_n\| (a \otimes \bar a) \xi_n -\xi_n\|_1=0$ for all $a\in \mathscr U (\A)$;
    \item $\mathbb E_ {\A\otimes 1} (\xi_n)=\mathbb E_{1\otimes \A}(\xi_n)=1$ for all $n$.
\end{enumerate}    
\end{definition}

\noindent Next we recall the following natural generalization of this weak compactness property to the case of extensions. This was also considered in the prior work \cite{CP11}.

\noindent To keep the $L^1$-formulation, we first specify the relative product. For an inclusion of abelian tracial von Neumann algebras $\A\subset\B$, denote by $\B\bar\otimes_\A\B$ the abelian von Neumann algebra obtained by the GNS construction from $\B\odot\B$ with respect to the state given by
\begin{equation*}
    \tau_\A(b\otimes c)=\tau(\mathbb E_\A(b)\mathbb E_\A(c)),\text{ for all } b,c\in\B.
\end{equation*}
We continue to write $b\otimes c$ for the corresponding relative elementary tensors. When $\A=L^\infty(Y,\nu)\subset\B=L^\infty(X,\mu)$, this is $L^\infty(X\times_YX,\mu\otimes_Y\mu)$, where $\mu\otimes_Y\mu=\int_Y\mu_y\otimes\mu_y\,d\nu(y)$ and $\mu=\int_Y\mu_y\,d\nu(y)$ is the disintegration over $Y$. The two canonical marginal maps are the trace-preserving conditional expectations given by
\begin{align*}
    \mathbb E_{\B\otimes1}(b\otimes c)=b\mathbb E_\A(c)\otimes1\text{ and }
    \mathbb E_{1\otimes\B}(b\otimes c)=1\otimes\mathbb E_\A(b)c.
\end{align*}
These extend to $L^1$, and the action on the relative product is $\sigma_g\otimes\sigma_g$. 

\noindent We also have a canonical unitary identification $L^2(\B\bar\otimes_\A\B,\tau_\A)\cong L^2(\langle\B,e_\A\rangle,\text{Tr})$ given by $ b\otimes c\longmapsto be_\A c$; here $\text{Tr}$ is the canonical semifinite trace, characterized by $\text{Tr}(be_\A c)=\tau(bc)$. Under this identification, multiplication by $b\otimes c$ corresponds to left multiplication by $b$ and right multiplication by $c$. The diagonal action corresponds to the extension $\widetilde\sigma_g(be_\A c)=\sigma_g(b)e_\A\sigma_g(c)$. Observe that this is an identification of Hilbert spaces, not of the two algebras or their $L^1$-spaces. All positivity, multiplication and marginal maps in the following definition are taken in the abelian relative product.

\begin{definition}\label{weaklycptexten} A trace preserving action $G\curvearrowright^\sigma (\A \subset \B)$ on an extension of abelian von Neumann algebras is called a weakly compact extension if one can find a net of positive, norm-one,  elements $(\xi_n)_n \subset L^1(\B\bar\otimes_\A\B,\tau_\A)$ satisfying  

\begin{enumerate}
    \item[(i)] $
    \lim_n\| \sigma_{g}\otimes \sigma_g(\xi_n)-\xi_n\|_1=0$ for all $g\in G$;
    \item[(ii)] $\lim_n\| (a \otimes \bar a) \xi_n -\xi_n\|_1=0$ for all $a\in \mathscr U (\B)$;
    \item[(iii)] $\mathbb E_ {\B\otimes 1} (\xi_n)=\mathbb E_{1\otimes \B}(\xi_n)=1$ for all $n$.
\end{enumerate}    
\end{definition}

\noindent Next we need the following result, which is a relative version of \cite[Proposition 3.2]{OP07}. For the reader's convenience we include a proof which is very similar with \cite[Proposition 3.2]{OP07}, keeping track of the relative basic construction and the normalization of the approximating elements.

\begin{theorem}\label{coamenability1} Consider a trace-preserving action $G\curvearrowright^\sigma (\A \subset \B)$ on an extension of abelian von Neumann algebras. Then the following are equivalent: 
\begin{enumerate}
    \item $G\curvearrowright^\sigma (\A \subset \B)$ is weakly compact;
    \item There exists a state $\varphi\,:\,\langle \B, e_\A \rangle \to \mathbb C$ that is $\B$-central, $G$-invariant, and satisfies $\varphi_{|\B}=\tau$; 
    \item The inclusion of von Neumann algebras $\A\rtimes G\subset \B\rtimes G$ is co-amenable in the sense of Popa.
\end{enumerate}   
\end{theorem}
\begin{proof}
    Let $\mathcal N=\langle\B,e_\A\rangle$ and let $\mathcal H=L^2(\B\bar\otimes_\A\B,\tau_\A)\cong L^2(\mathcal N,\text{Tr})$, via the identification described above.

    \noindent (1)$\Rightarrow$ (2) Fix a net $(\eta_n)_n\subset L^1(\B\bar\otimes_\A\B,\tau_\A)$ satisfying conditions (i)--(iii) of Definition \ref{weaklycptexten} and let $\xi_n=\eta_n^{1/2}\in\mathcal H$ for each $n$, where the square root is taken in the abelian relative product. For every $g\in G$ and $a\in\mathscr U(\B)$, we have
    \begin{align*}
        \| (\sigma_g\otimes\sigma_g)(\xi_n)-\xi_n\|_2^2
        &\leq\| (\sigma_g\otimes\sigma_g)(\eta_n)-\eta_n\|_1\text{ and }
        \|(a\otimes\bar a)\xi_n-\xi_n\|_2^2
    \leq 2\|(a\otimes\bar a)\eta_n-\eta_n\|_1.
    \end{align*}
    The first inequality follows from the inequality for square roots of nonnegative functions, and the second follows from $|z-1|^2\leq2|z-1|$ for $|z|=1$. For every $n$ consider the  state $\varphi_n:\mathcal N\rightarrow\mathbb C$ by $\varphi_n(x)=\langle x\xi_n,\xi_n\rangle_{\mathcal H}$. Using the Cauchy--Schwarz inequality, for every $x\in\mathcal N$, $a\in\mathscr U(\B)$ and $g\in G$, we have 
    \begin{align*}
        |\varphi_n(axa^*)-\varphi_n(x)|
        &\leq2\|x\|_\infty\|(a^*\otimes a)\xi_n-\xi_n\|_2, \text{ and}\\
        |\varphi_n(\widetilde\sigma_g(x))-\varphi_n(x)|
        &\leq2\|x\|_\infty\|(\sigma_{g^{-1}}\otimes\sigma_{g^{-1}})(\xi_n)-\xi_n\|_2.
    \end{align*}
    Thus any weak-$*$ cluster point $\varphi$ of $(\varphi_n)_n$ is $\B$-central and $G$-invariant. Moreover, (iii) implies that for any $b\in\B$, we have 
 $\varphi_n(b)=\tau_\A((b\otimes1)\eta_n)
        =\tau\bigl(b\mathbb E_{\B\otimes1}(\eta_n)\bigr)=\tau(b)$;  hence $\varphi|_\B=\tau$.

    \bigskip

    \noindent (2) $\Rightarrow$ (1) Let $\varphi$ be a state satisfying (2). As normal states are weak-$*$ dense in the space of states on $\mathcal N$, one can find a net of elements $T_n\in L^1(\mathcal N,\text{Tr})_+$ with $\text{Tr}(T_n)=1$ so that for every $x\in \mathcal N$ we have 
    \begin{equation*}
        \lim_n\text{Tr}(T_nx)=\varphi(x).
    \end{equation*}
    Let $r_n\in L^1(\B)_+$ be the density of the restriction to $\B$, so that for every $b\in \B$ we have 
    \begin{equation*}
        \text{Tr}(T_nb)=\tau(r_nb).
    \end{equation*}
    The assumptions on $\varphi$ imply that $r_n-1$, $\widetilde\sigma_g(T_n)-T_n$ and $aT_na^*-T_n$ converge weakly to zero in their respective $L^1$-spaces, for every $g\in G$ and $a\in\mathscr U(\B)$. Applying Hahn--Banach separation to finite tuples of these differences and taking convex combinations, we can assume for all $g\in G$ and $a\in \mathscr U(\B)$ we have 
    \begin{align*}
        \lim_n\|r_n-1\|_1&=0,\\
        \lim_n\|\widetilde\sigma_g(T_n)-T_n\|_1&=0,\\
        \lim_n\|aT_na^*-T_n\|_1&=0.
    \end{align*}
    Here and below, nets may be reindexed after taking convex combinations.
\vskip 0.05in
   \noindent To make the restriction to $\B$ exactly equal to $\tau$ we have to do a truncation.  So consider the following elements:
 $c_n=\max\{1,r_n\}^{-1/2}$, $d_n=(1-r_n)_+$, and $S_n=c_nT_nc_n+d_n^{1/2}e_\A d_n^{1/2}$. Then  $c_n,d_n\in \B$, and for every $b\in\B$, we have 
  ${\rm Tr}(S_nb)=\tau\bigl((c_n^2r_n+d_n)b\bigr)=\tau(b)$.
   Moreover, using the Cauchy--Schwarz inequality (in  noncommutative $L^p$-spaces), we have 
    \begin{align*}
        \|T_n-c_nT_nc_n\|_1
        \leq2\|T_n^{1/2}\|_2\|(1-c_n)T_n^{1/2}\|_2
        =2\tau\bigl((1-c_n)^2r_n\bigr)^{1/2}
        \leq2\|r_n-1\|_1^{1/2}.
    \end{align*}
    Since $\|d_n^{1/2}e_\A d_n^{1/2}\|_1=\tau(d_n)\leq\|r_n-1\|_1$, we obtain
    \begin{equation*}
        \|S_n-T_n\|_1\leq2\|r_n-1\|_1^{1/2}+\|r_n-1\|_1\to0.
    \end{equation*}
    Consequently, $\text{Tr}(S_n)=1$ and both approximate invariance properties above remain true with $S_n$ in place of $T_n$.

    \noindent Let $\zeta_n=S_n^{1/2}\in L^2(\mathcal N,\text{Tr})$. By Powers--St\o rmer inequality, for every $g\in G$ and $a\in\mathscr U(\B)$, we have
    \begin{align*}
        \|\widetilde\sigma_g(\zeta_n)-\zeta_n\|_2^2
        &\leq\|\widetilde\sigma_g(S_n)-S_n\|_1\to0, \text{ and}\\
        \|a\zeta_na^*-\zeta_n\|_2^2
        &\leq\|aS_na^*-S_n\|_1\to0.
    \end{align*}
     Denote by $\xi_n\in\mathcal H$ the vector corresponding to $\zeta_n$ and set $\eta_n=|\xi_n|^2\in L^1(\B\bar\otimes_\A\B,\tau_\A)_+$, taking the absolute value and the square in the abelian relative product. Since $\|\xi_n\|_2=1$, the Cauchy--Schwarz inequality yields
    \begin{align*}
        \|(\sigma_g\otimes\sigma_g)(\eta_n)-\eta_n\|_1
        &\leq2\|(\sigma_g\otimes\sigma_g)(\xi_n)-\xi_n\|_2\to0,\\
        \|(a\otimes\bar a)\eta_n-\eta_n\|_1
        &\leq\|(a\otimes\bar a)\xi_n-\xi_n\|_2\to0.
    \end{align*}
    Finally, the left and right actions under the Hilbert space identification, together with $\zeta_n=\zeta_n^*$, give for every $b\in\B$, that
    \begin{align*}
        \tau_\A((b\otimes1)\eta_n)
        &=\langle b\zeta_n,\zeta_n\rangle=\text{Tr}(S_nb)=\tau(b),\\
        \tau_\A((1\otimes b)\eta_n)
        &=\langle\zeta_nb,\zeta_n\rangle=\text{Tr}(S_nb)=\tau(b).
    \end{align*}
    Thus $\mathbb E_{\B\otimes1}(\eta_n)=\mathbb E_{1\otimes\B}(\eta_n)=1$. Altogether, these show $(\eta_n)_n$ satisfies (i)--(iii) in Definition \ref{weaklycptexten}, as desired.

    \bigskip

    \noindent (2) $\Rightarrow$ (3) Under the usual identification $L^2(\B\rtimes G)\cong L^2(\B)\otimes\ell^2(G)$ via the map $\widehat{bu_g}\mapsto\hat b\otimes\delta_g$ for $b\in\B$, $g\in G$, observe that $e_{\A\rtimes G}$ is mapped to $e_\A\otimes1_{\ell^2(G)}$. Let $U_g\hat b=\widehat{\sigma_g(b)}$. Consequently, on the Hilbert space $L^2(\B)\otimes\ell^2(G)$, we have the following identifications
    \begin{align*}
        \B\rtimes G&=
        \bigl(\B\otimes1_{\ell^2(G)}\cup\{U_g\otimes\lambda(g)\,|\,g\in G\}\bigr)'',\\
        \langle\B,e_\A\rangle\otimes1_{\ell^2(G)}
        &=\bigl(\B\otimes1_{\ell^2(G)}\cup\{e_\A\otimes1_{\ell^2(G)}\}\bigr)'',\\
        \langle\B\rtimes G,e_{\A\rtimes G}\rangle
        &=\bigl(\B\otimes1_{\ell^2(G)}\cup\{U_g\otimes\lambda(g)\,|\,g\in G\}
        \cup\{e_\A\otimes1_{\ell^2(G)}\}\bigr)''.
    \end{align*}
    Since the action $G\curvearrowright^\sigma\B$ preserves $\A$, the unitaries $u_g=U_g\otimes\lambda(g)$ fix $e_\A\otimes1_{\ell^2(G)}$ by conjugation and implement the action $\widetilde\sigma$ on $\mathcal N$. This shows that
    \begin{equation*}
        \langle\B\rtimes G,e_{\A\rtimes G}\rangle\cong\mathcal N\rtimes_{\widetilde\sigma}G.
    \end{equation*}
    Let $\mathbb E_{\mathcal N}:\mathcal N\rtimes_{\widetilde\sigma}G\to\mathcal N$ be the canonical conditional expectation and let $\widetilde\varphi=\varphi\circ\mathbb E_{\mathcal N}$. For every $T\in\mathcal N\rtimes_{\widetilde\sigma}G$, $b\in\B$ and $g\in G$, we have
    \begin{align*}
        \widetilde\varphi(bT)
        &=\varphi(b\mathbb E_{\mathcal N}(T))
        =\varphi(\mathbb E_{\mathcal N}(T)b)=\widetilde\varphi(Tb), \text{ and}\\
        \widetilde\varphi(u_gTu_g^*)
        &=\varphi(\widetilde\sigma_g(\mathbb E_{\mathcal N}(T)))
        =\varphi(\mathbb E_{\mathcal N}(T))=\widetilde\varphi(T).
    \end{align*}
    Also, ${\mathbb E_{\mathcal N}}_{|\B\rtimes G}=\mathbb E_\B$, and hence $\widetilde\varphi|_{\B\rtimes G}=\tau$. The displayed identities give centrality on the algebraic crossed product. Since for every $x\in\B\rtimes G$ we have $|\widetilde\varphi(xT)-\widetilde\varphi(Tx)|\leq2\|T\|_\infty\|x\|_2$, then the Cauchy--Schwarz inequality and $\|\cdot\|_2$-density extend centrality to $\B\rtimes G$. Thus $\A\rtimes G\subset\B\rtimes G$ is co-amenable.

    \bigskip

    \noindent (3) $\Rightarrow$ (2) Suppose $\tilde\varphi$ is a $\B\rtimes G$-central state on $\langle\B\rtimes G,e_{\A\rtimes G}\rangle$ so that $\tilde\varphi|_{\B\rtimes G}=\tau$. Under the preceding identification, its restriction to $\langle\B,e_\A\rangle$ satisfies the hypotheses of (2).
\end{proof}

\begin{corollary}\label{compactcoamenable} Assume the trace preserving action $G\curvearrowright^\sigma (\A \subset \B)$ on an extension of abelian von Neumann algebras is compact. Then the extension is weakly compact and the inclusion of von Neumann algebras $\A\rtimes G\subseteq\B\rtimes G$ is co-amenable.
\end{corollary}
\begin{proof}
    Fix a finite set $\mathcal E\subset\mathscr U(\B)$ and $\varepsilon>0$. By Definition \ref{compactext}, there exist a finite partition of unity by projections $p_1,\dots,p_k\in\B$ and $\mathcal C=(\A\cup\{p_1,\dots,p_k\})''$ such that
    \begin{equation}\label{eq:firstestimate}
        \sup_{a\in\mathcal E,\,g\in G}\|\sigma_g(a)-\mathbb E_{\mathcal C}(\sigma_g(a))\|_2<\varepsilon.
    \end{equation}
    Indeed, first approximate the elements of $\mathcal E$ by the linear span of the set $\mathcal F$ in Definition \ref{compactext}, and then approximate the finitely many vectors $\xi_i$ occurring there by simple functions on a common finite partition. The uniform bounds on the coefficients in $\A$ give the displayed estimate.
\vskip 0.05in
    \noindent Denote by $h_j:=\mathbb E_\A(p_j)$ and, in the abelian relative product, consider $\eta_0:=\sum_{j=1}^k(h_j^{-1}\otimes1)(p_j\otimes p_j)$; here the inverse  $h_j^{-1}$ is taken on the support $s_j={\rm supp}(h_j)\in \A$  and it is zero on its complement. Then $\eta_0\geq0$. Moreover, using the formula $\mathbb E_{\B\otimes 1}(x\otimes y)= x \mathbb E_\A(y)\otimes 1$  we see that  
    \begin{align*}
    \mathbb E_ {\B\otimes 1}(\eta_0)= \sum_j  h_j^{-1} p_j \mathbb E_{\A}(p_j) \otimes 1= \sum_j s_j p_j\otimes 1=\sum_j p_j\otimes 1 =1. 
        \end{align*}
Similarly $\mathbb E_{1\otimes \B}(\eta_0)=1$. Also using orthogonality and similar computations we can see that 
    \begin{align*}
        \|\eta_0\|_2^2=  \sum_{j=1}^k
        \tau_\A ((  h_j^{-2} p_j)\otimes p_j)= \sum_{j=1}^k
        \tau_\A ((  h_j^{-2} p_j)\otimes p_j)= \sum^k_{j=1} = \tau( h_j^{-2}\mathbb E_\A(p_j) \mathbb E_\A(p_j))=\sum^k_{j=1}\tau(s_j)\leq k.
    \end{align*}
    For every $b\in \B$, using the relative-product trace
$\tau_\A(x\otimes y)=
\tau(\mathbb E_\A(x)\mathbb E_\A(y))$, we obtain
\begin{equation}\begin{split}\label{eq:secondidentity}
    \|(b\otimes1-1\otimes b)\eta_0^{1/2}\|_2^2&=
    \tau_\A\left(
        \bigl(|b|^2\otimes1+1\otimes|b|^2
        -b^*\otimes b-b\otimes b^*\bigr)\eta_0
    \right)\\
    &\quad=
    \sum_{j=1}^k\tau_\A\Bigl(
        (h_j^{-1}\otimes1)
        \bigl(|b|^2p_j\otimes p_j
        +p_j\otimes|b|^2p_j-b^*p_j\otimes bp_j
        -bp_j\otimes b^*p_j\bigr)
    \Bigr)\\
    &\quad=
    \sum_{j=1}^k\tau\left(
        h_j^{-1}
        \bigl(\mathbb E_\A(|b|^2p_j) h_j+h_j\mathbb E_\A(|b|^2p_j) - \mathbb E_\A(b^*p_j)\mathbb E_\A(bp_j)-\mathbb E_\A(bp_j)\mathbb E_\A(b^*p_j)\bigr)
    \right)\\
    &\quad=
    2\sum_{j=1}^k\tau(\mathbb E_\A(|b|^2p_j))
    -2\sum_{j=1}^k\tau(h_j^{-1}|\mathbb E_\A(bp_j)|^2)\\
    &\quad=
    2\tau(|b|^2)-2\|\mathbb E_{\mathcal C}(b)\|_2^2\\
    &\quad=
    2\|b-\mathbb E_{\mathcal C}(b)\|_2^2.
\end{split}\end{equation}
The trace formulas involving negative powers of $h_j$ follow
by truncation. The last equality follows because
$\mathbb E_{\mathcal C}\,:\,L^2(\B) \to L^2(\mathcal C)$ is the orthogonal projection.
\vskip 0.05in

\noindent Fix $a\in\mathcal E\subset\mathscr U(\B)$ and $g\in G$. Then, using the invariance of $\tau_\A$ under the diagonal action, the Cauchy--Schwarz inequality, \eqref{eq:secondidentity}, and \eqref{eq:firstestimate} we obtain
\begin{align*}
    \|(a\otimes\bar a-1)(\sigma_g\otimes\sigma_g)(\eta_0)\|_1
    &\quad=\|(\sigma_{g^{-1}}(a)\otimes\overline {\sigma_{g^{-1}}(a)}-1)\eta_0\|_1\\
    &\quad=\|(\sigma_{g^{-1}}(a)\otimes1-1\otimes \sigma_{g^{-1}}(a))\eta_0\|_1\\
    &\quad=\|\bigl(\sigma_{g^{-1}}(a)\otimes1-1\otimes \sigma_{g^{-1}}(a))\eta_0^{1/2}\bigr)
        \eta_0^{1/2}\|_1\\
    &\quad\leq\|(\sigma_{g^{-1}}(a)\otimes1-1\otimes \sigma_{g^{-1}}(a))\eta_0^{1/2}\|_2
        \|\eta_0^{1/2}\|_2\\
    &\quad=2^{1/2}\|\sigma_{g^{-1}}(a)
        -\mathbb E_{\mathcal C}(\sigma_{g^{-1}}(a))\|_2\\
    &\quad<2^{1/2}\varepsilon.
\end{align*}

    \noindent Let $\eta$ be the unique element of minimal $L^2$-norm in the $L^2$-closed convex hull of $\{(\sigma_g\otimes\sigma_g)(\eta_0)\,|\,g\in G\}$. Uniqueness implies that $\eta$ is $G$-invariant. Positivity, both marginal identities and the last $L^1$-estimate pass to this closed convex hull. Consequently $\eta\geq0$, $\|\eta\|_1=1$, both marginals are $1$, and $\|(a\otimes\bar a)\eta-\eta\|_1\leq2^{1/2}\varepsilon$ for every $a\in\mathcal E$. Directing the pairs $(\mathcal E,\varepsilon)$ by inclusion of $\mathcal E$ and decrease of $\varepsilon$ gives a net satisfying Definition \ref{weaklycptexten}. We leave the details to the reader.  The co-amenability conclusion now follows from Theorem \ref{coamenability1}.
\end{proof}

\begin{theorem}\label{coamenabilityfinitechain} Let $\A_1\subseteq\cdots\subseteq\A_n$ be abelian von Neumann algebras and let $G\curvearrowright^\sigma\A_n$ be a trace preserving action leaving each $\A_i$ invariant. Assume the extension $G\curvearrowright^\sigma(\A_i\subset\A_{i+1})$ is weakly compact for all $1\leq i<n$. Then so is $G\curvearrowright^\sigma(\A_1\subset\A_n)$.
\end{theorem}
\begin{proof} Denote by $\M_i=\mathcal A_i\rtimes G$ for all $1\leq i\leq n$. From the assumptions and Theorem \ref{coamenability1} we have that $\M_{i+1}$ is amenable relative to $\M_i$ inside $\M_{i+1}$. This also holds inside $\M_n$: compress the basic construction $\langle\M_n,e_{\M_i}\rangle$ by $e_{\M_{i+1}}$ and compose with an $\M_{i+1}$-central state on $\langle\M_{i+1},e_{\M_i}\rangle$. Its restriction to $\M_n$ is $\tau\circ\mathbb E_{\M_{i+1}}=\tau$. Using iteratively \cite[Proposition 2.4(3)]{OP07} we get that $\M_n$ is amenable relative to $\M_1$ inside $\M_n$. Then Theorem \ref{coamenability1} above yields the conclusion.
\end{proof}

\begin{theorem}\label{weaklycompactlimit} Let $\mathcal A_0\subseteq\mathcal A_1\subseteq\mathcal A_2\subseteq\cdots$ be abelian von Neumann algebras and put $\A=\overline{\bigcup_m\mathcal A_m}^{\rm WOT}$. Let $G\curvearrowright^\sigma\A$ be a trace preserving action leaving each $\A_n$ invariant. Assume the extension $G\curvearrowright^\sigma(\A_0\subset\A_n)$ is weakly compact for every $n\geq1$. Then so is $G\curvearrowright^\sigma(\A_0\subset\A)$. In particular, if every consecutive extension $G\curvearrowright^\sigma(\A_n\subset\A_{n+1})$, $n\geq1$, is weakly compact, then $G\curvearrowright^\sigma(\A_1\subset\A)$ is weakly compact.
\end{theorem}
\begin{proof}
    Denote by $\M_i=\mathcal A_i\rtimes G$ for all $i\geq0$ and let $\M=\A\rtimes G$. From the assumptions and Theorem \ref{coamenability1} we have that $\M_n$ is amenable relative to $\M_0$ inside $\M_n$, and hence inside $\M$ by the compression argument in the preceding proof. Then \cite[Lemma 2.7]{DHI16} gives that $\M=\overline{\bigcup_n\M_n}^{\rm WOT}$ is amenable relative to $\M_0$ inside $\M$. Using Theorem \ref{coamenability1}, this yields the first conclusion. The final assertion follows by applying Theorem \ref{coamenabilityfinitechain} first.
\end{proof}

\noindent The following von Neumann algebraic characterization of distal towers in terms of quasinormalizers of subalgebra inclusions was first obtained in an unpublished work of the first author and J. Peterson \cite{CP11}. A detailed proof can also be found in \cite[Corollary 4.25]{BCCMSW24}. To properly recall the result we first recall the notion of quasinormalizers.

\begin{definition}
    Given an inclusion $\mathcal N\subseteq \M$ of von Neumann algebras, the \emph{quasinormalizer} $\mathcal{QN}(\mathcal N\subseteq \M)$ is the $*$-subalgebra of $\M$ consisting of all $x\in \M$ for which there are finitely many $x_1,\dots,x_k\in \M$ satisfying $\mathcal N x\subset \sum_i x_i\mathcal N$ and $x\mathcal N\subset \sum_i \mathcal Nx_i$.
\end{definition}
\begin{theorem}\label{towervna}
    Let $G\curvearrowright X$ be a pmp ergodic action and let $\{G \curvearrowright (X_\alpha,\mu_\alpha)\}_{\alpha\leq \eta}$ be the canonical Furstenberg--Zimmer tower. Consider the associated crossed product von Neumann algebras $\M=L^\infty(X)\rtimes G$ and $\M_\alpha=L^\infty(X_\alpha)\rtimes G$ for each ordinal $\alpha\leq \eta$. The following hold:
    \begin{enumerate}
        \item For all $\alpha\le \beta\le \eta$, we have the inclusions $\mathcal L(G)=\M_0\subseteq \M_\alpha\subseteq \M_\beta\subseteq \M_\eta\subseteq \M$.
        \item For all $\alpha<\eta$, we have that $\mathcal{QN}(\M_\alpha\subseteq\M)''=\M_{\alpha+1}$. At the terminal stage, $\mathcal{QN}(\M_\eta\subseteq\M)''=\M_\eta$.
        \item For any limit ordinal $\beta\leq \eta$, we have that $\overline{\bigcup_{\alpha< \beta}L^\infty(X_\alpha)}^{\rm WOT}=L^\infty(X_\beta)$ and $\overline{\bigcup_{\alpha< \beta}\M_\alpha}^{\rm WOT}=\M_\beta$.
        \item One can find an infinite sequence $(g_n)_n\subset G$ so that for all $x,y\in L^\infty(X)\ominus L^\infty (X_\eta)$, we have
        \begin{equation*}
            \lim_{n\rightarrow\infty}|| \mathbb E_{L^\infty (X_\eta)}(x\sigma_{g_n}(y))||_2=0.
        \end{equation*}
    \end{enumerate}
\end{theorem}


\noindent Finally, as a consequence of the prior results, we record the following corollary, which was also obtained in the previous work, \cite{CP11}.

\begin{corollary}\label{cor: dital implies weak compact}
    Any free, ergodic, distal action $G \curvearrowright (X, \mu)$ is weakly compact.
\end{corollary}

\begin{proof} Let $(X_\alpha)_{\alpha\leq\eta}$ be the canonical distal tower and let $\A_\alpha=L^\infty(X_\alpha)$. We prove by transfinite induction that $G\curvearrowright(\mathbb C\subset\A_\alpha)$ is weakly compact. At $\alpha=0$ this is immediate. At a successor stage, the extension $G\curvearrowright(\A_\alpha\subset\A_{\alpha+1})$ is compact, and hence weakly compact by Corollary \ref{compactcoamenable}. Then Theorem \ref{coamenabilityfinitechain} yields the inductive step. At a nonzero limit ordinal $\lambda\leq\eta$, choose an increasing cofinal sequence $(\alpha_n)_n$ in $\lambda$, which is possible since $\eta$ is countable. Since $\A_\lambda=\overline{\bigcup_n\A_{\alpha_n}}^{\rm WOT}$, then Theorem \ref{weaklycompactlimit}, with fixed base $\mathbb C$, gives the limit step. Finally, distality implies that $X_\eta=X$, and the definition for the base $\mathbb C$ is precisely the ordinary weak compactness definition. \end{proof}

\section{Proofs of the main results}

In this section we prove the main result of the paper, namely Theorem \ref{mainthm:1}.
\begin{theorem}\label{thm: weak compact implies isp} Let $G$ be a non-inner amenable group and let $G\curvearrowright^\sigma (\A, \tau)$ be a weakly compact trace-preserving action on an abelian von Neumann algebra $\A$. Then for any intermediate von Neumann subalgebra $\eL(G)\subseteq \Q \subseteq \A\rtimes G$ there is a $G$-invariant von Neumann subalgebra $\B\subseteq \A$ such that $\Q= \B\rtimes G$. 
\end{theorem}

\begin{proof} Let $(\xi_n)_n\subset L^1(\A\bar{\otimes}\A)$ be a net of positive, unital elements satisfying the conditions that make the $G\curvearrowright(\A,\tau)$ weakly compact. Since $\mathbb E_{\M \bar\otimes \Q}$ is $\|\cdot \|_1$-contractive, $\M\bar\otimes \Q$-bimodular and $\eL(G)\subset \Q$, condition (i) for $\xi_n$ implies that 

\begin{equation}\label{specgap0} 
    \lim_n\| (u_g \otimes u_g)\mathbb E_{\M \bar\otimes \Q}(\xi_n) (u_{g^{-1}}\otimes u_{g^{-1}})- \mathbb E_{\M \bar\otimes \Q}(\xi_n)\|_1=0
\end{equation}

\noindent Letting $\eta_n =\mathbb E_{\M \bar\otimes \Q}(\xi_n)^{1/2} \in L^2 (\M\bar\otimes \Q )$, and using relation \eqref{specgap0} together with Powers-St\o rmer inequality we further get

\begin{equation}\label{specgap5}
    \lim_n\| (u_g \otimes u_g) \eta_n(u_{g^{-1}}\otimes u_{g^{-1}})- \eta_n \|_2=0.
\end{equation}

\noindent Since $G$ is non-inner amenable, by Corollary \ref{corspectralgap} there exists a finite set $K\subset G$ and $S>0$ such that for all $\xi\in L^2(\M\bar \otimes \M)$ we have 

\begin{equation*}
    \sum_{g\in K}\| (u_g \otimes u_g)\xi -\xi (u_g \otimes u_g) \|^2_2\geq S\|\xi- \mathbb E_{\A \bar\otimes \A }(\xi)\|^2_2.
\end{equation*}      

\noindent Letting $\xi=\eta_n$ in the prior relation and using \eqref{specgap5}, we further get that 

\begin{equation}\label{el2close}
    \lim_n\|\eta_n- \mathbb E_{\A \bar\otimes \A }(\eta_n)\|_2=0.
\end{equation}
We continue with the following. 
\begin{claim}\label{c1}
\begin{equation}\label{closenorm1}
    \lim_n\| \mathbb E_{\M \bar\otimes \A }\circ \mathbb E_{\M \bar\otimes \Q}(\xi_n) - \mathbb E_{\M \bar\otimes \Q}(\xi_n) \|_1=0.
\end{equation}
\end{claim}

\noindent\emph{Proof of Claim \ref{c1}.} Let $\mu_n\in \M \bar\otimes \M$ such that $\lim_n\|\mu_n-\eta_n\|_2=0$. From \eqref{el2close} we obtain 

\begin{equation}\label{ineq1}
    \lim_n\|\mu_n- \mathbb E_{\A \bar\otimes \A }(\mu_n)\|_2=0.
\end{equation}

\noindent Consider the partial isometry $v_n\in \M \bar\otimes \M$ satisfying $|\mathbb E_{\M \bar\otimes \A}(\mu_n^2)-\mu_n^2|= v_n(\mathbb E_{\M \bar\otimes \A}(\mu_n^2)-\mu_n^2)$. Notice that 

\begin{equation}\label{eq4}\begin{split}
    \|\mathbb E_{\M \bar\otimes \A}(\mu_n^2)-\mu_n^2\|_1&= \tau(v_n (\mathbb E_{\M \bar\otimes \A}(\mu_n^2)-\mu_n^2))
    =  \tau(w_n(\mathbb E_{\M \bar\otimes \A}(\mu_n^2)-\mu_n^2)),
    \end{split}
\end{equation} 

\noindent where $w_n = v_n - \mathbb E_{\M\bar\otimes \A}(v_n)$. Using algebraic manipulations we see that 

\begin{align}\label{ineq2}
    \tau(w_n(\mathbb E_{\M \bar\otimes \A}(\mu_n^2)-\mu_n^2))&= \tau(w_n (\mathbb E_{\M\bar\otimes \A }(\mu_n^2) -(\mu_n - \mathbb E_{\M\bar\otimes \A}(\mu_n))^2 -\mathbb E_{\M\bar\otimes \A}(\mu_n)(\mu_n - \mathbb E_{\M\bar\otimes \A}(\mu_n))\nonumber\\
    &\quad\quad\quad- (\mu_n - \mathbb E_{\M\bar\otimes \A} (\mu_n))\mathbb E_{\M\bar\otimes \A}(\mu_n)-\mathbb E_{\M\bar\otimes \A}(\mu_n)^2)).
\end{align}

\noindent Furthermore, using that $\tau(w_n x)=0$ for all $x\in \M \bar\otimes \A$ and Cauchy-Schwarz inequality, we see the previous equality gives

\begin{equation}\label{ineq3}\begin{split}
    \tau(-w_n&(\mu_n - \mathbb E_{\M\bar\otimes \A}(\mu_n))^2 - w_n\mathbb E_{\M\bar\otimes \A}(\mu_n)(\mu_n - \mathbb E_{\M\bar\otimes \A}(\mu_n)) - w_n(\mu_n - \mathbb E_{\M\bar\otimes \A} (\mu_n))\mathbb E_{\M\bar\otimes \A}(\mu_n))\\
    &\leq \|w_n (\mu_n - \mathbb E_{\M\bar\otimes \A}(\mu_n))\|_2 \|\mu_n - \mathbb E_{\M\bar\otimes \A}(\mu_n)\|_2 + \|w_n\mathbb E_{\M\bar\otimes \A}(\mu_n)\|_2\| \mu_n - \mathbb E_{\M\bar\otimes \A}(\mu_n)\|_2\\
    &\quad\quad +\|\mathbb E_{\M\bar\otimes \A}(\mu_n)w_n\|_2 \| \mu_n - \mathbb E_{\M\bar\otimes \A}(\mu_n)\|_2 \\
    &\leq \|w_n\|_\infty\left (\|\mu_n - \mathbb E_{\M\bar\otimes \A}(\mu_n)\|_2^2+2\|\mu_n\|_2\| \mu_n - \mathbb E_{\M\bar\otimes \A}(\mu_n)\|_2\right ) \end{split}
\end{equation}

\noindent As $\|w_n\|_\infty \leq 2$ and $\|\mu_n\|_2\leq 1$ for all $n$, altogether, relations \eqref{ineq1}, \eqref{eq4}, \eqref{ineq2} and \eqref{ineq3} yield

\begin{equation*}\label{closenorm1'}
    \lim_n\| \mathbb E_{\M \bar\otimes \A }(\mu_n^2 )- \mu_n^2 \|_1=0.
\end{equation*}

\noindent Finally, since $\lim_n\|\mu_n -\eta_n\|_2=0$, the prior inequality and Powers-St\o rmer's inequality gives \eqref{closenorm1}. $\hfill\blacksquare$
\vskip 0.05in
\noindent Next we prove the following.

\begin{claim}\label{c2} For every $a\in \A$ we have that $\mathbb E_\Q(a)= \mathbb E_{\Q\cap \A}(a)$.
    
\end{claim}

\noindent\emph{Proof of Claim \ref{c2}.} Since $\A$ is the norm-closed linear span of its self-adjoint
unitaries, it suffices to show that for every
$a\in\mathscr U(\A)$ with $a=a^*$ and every $\varepsilon>0$,
we have
$\|\mathbb E_\Q(a)-\mathbb E_{\Q\cap\A}(a)\|_2\leq\varepsilon$.

\vskip 0.05in
\noindent By the von Neumann-Halperin projections theorem \cite{vN50,Ha62}, there exists $k\in \mathbb N$ such that 

\begin{equation}\label{convpingpong}
    \|(\mathbb E_\Q \circ \mathbb E_\A \circ \mathbb E_\Q)^k(a)-\mathbb E_{\Q \cap \A}(a)\|_2\leq \varepsilon^2.
\end{equation}

\noindent Using the triangle inequality and the $\tau$-invariance of conditional expectations $\mathbb E_{\M\bar\otimes \Q}$ and $\mathbb E_{\M\bar\otimes \A}$, we see for any $n$ we have

\begin{equation*}\begin{split}
    |\tau((a \otimes \mathbb E_\Q(a))\xi_n )|& = |\tau((a \otimes \mathbb E_\Q(a))\mathbb E_{\M \bar \otimes \Q}(\xi_n) )| \\
    & \leq |\tau( (a \otimes \mathbb E_\Q(a)) \mathbb E_{\M \bar\otimes \A} \circ \mathbb E_{\M \bar \otimes \Q}(\xi_n) )|  + \|\mathbb E_{\M \bar\otimes \A} \circ \mathbb E_{\M \bar \otimes \Q}(\xi_n)-\mathbb E_{\M \bar \otimes \Q}(\xi_n)\|_1 \\ 
    & = |\tau( (a \otimes \mathbb E_\A\circ \mathbb E_\Q(a) )\mathbb E_{\M \bar \otimes \Q}(\xi_n) )|  + \|\mathbb E_{\M \bar\otimes \A} \circ \mathbb E_{\M \bar \otimes \Q}(\xi_n)-\mathbb E_{\M \bar \otimes \Q}(\xi_n)\|_1 \\ 
    &=|\tau( (a \otimes \mathbb E_\Q\circ \mathbb E_\A\circ \mathbb E_\Q(a) ) \xi_n)| + \|\mathbb E_{\M \bar\otimes \A} \circ \mathbb E_{\M \bar \otimes \Q}(\xi_n)-\mathbb E_{\M \bar \otimes \Q}(\xi_n)\|_1.\\    \end{split}
\end{equation*}

\noindent Using this argument $k-1$ more times we get the following inequality 

\begin{equation*}\label{ineq6}\begin{split}
    |\tau((a \otimes \mathbb E_\Q(a))\xi_n) |\leq |\tau((a \otimes (\mathbb E_\Q\circ \mathbb E_\A\circ \mathbb E_\Q)^k(a))\xi_n)| + k \|\mathbb E_{\M \bar\otimes \A} \circ \mathbb E_{\M \bar \otimes \Q}(\xi_n)-\mathbb E_{\M \bar \otimes \Q}(\xi_n)\|_1.\end{split}
\end{equation*}

\noindent Using condition (iii) for $\xi_n$, the triangle inequality, \eqref{convpingpong} and the prior inequality, we see that 

\begin{equation*}\begin{split}
    \tau (\mathbb E_\Q(a)a^*)&= \tau ( (1\otimes  ( \mathbb E_\Q(a) a^*))\xi_n)\\
    &\leq |\tau( (a \otimes \mathbb E_\Q(a)) \xi_n )| + \| \xi_n -(a^*\otimes a^*) \xi_n\|_1\\
    & \leq |\tau( (a \otimes (\mathbb E_\Q\circ \mathbb E_\A\circ \mathbb E_\Q)^k(a) ) \xi_n)|  + k \|\mathbb E_{\M \bar\otimes \A} \circ \mathbb E_{\M \bar \otimes \Q}(\xi_n)-\mathbb E_{\M \bar \otimes \Q}(\xi_n)\|_1+ \| \xi_n -(a^*\otimes a^*) \xi_n\|_1\\
    & \leq |\tau( (1 \otimes (\mathbb E_\Q\circ \mathbb E_\A\circ \mathbb E_\Q)^k(a) a^*) \xi_n)|  + k \|\mathbb E_{\M \bar\otimes \A} \circ \mathbb E_{\M \bar \otimes \Q}(\xi_n)-\mathbb E_{\M \bar \otimes \Q}(\xi_n)\|_1+ 2\| \xi_n -(a^*\otimes a^*) \xi_n\|_1\\
    & =|\tau( (\mathbb E_\Q\circ \mathbb E_\A\circ \mathbb E_\Q)^k(a) a^*)|  + k \|\mathbb E_{\M \bar\otimes \A} \circ \mathbb E_{\M \bar \otimes \Q}(\xi_n)-\mathbb E_{\M \bar \otimes \Q}(\xi_n)\|_1+ 2\| \xi_n -(a^*\otimes a^*) \xi_n\|_1\\
    & =|\tau( \mathbb E_{\Q\cap \A}(a) a^*)|  + k \|\mathbb E_{\M \bar\otimes \A} \circ \mathbb E_{\M \bar \otimes \Q}(\xi_n)-\mathbb E_{\M \bar \otimes \Q}(\xi_n)\|_1+ 2\| \xi_n -(a^*\otimes a^*) \xi_n\|_1  \\&\quad +\|(\mathbb E_\Q \circ \mathbb E_\A \circ \mathbb E_\Q)^k(a)-\mathbb E_{\Q\cap \A}(a)\|_2 \\
    &\leq \tau( \mathbb E_{\Q\cap \A}(a) a^*)  + k \|\mathbb E_{\M \bar \otimes \A} \circ \mathbb E_{\M \bar \otimes \Q}(\xi_n)-\mathbb E_{\M \bar \otimes \Q}(\xi_n)\|_1+ 2\| \xi_n -(a^*\otimes a^*) \xi_n\|_1 + \varepsilon^2.   
    \end{split}
\end{equation*}

\noindent The previous inequality shows that for all $n$ we have 

\begin{align*}
  \| \mathbb E_{\Q}(a)-\mathbb E_{\Q\cap\A}(a)\|_2^2&=|\tau(\mathbb E_{\Q}(a)a^*)-\tau(\mathbb E_{\Q\cap\A}(a)a^*)|\\
  &\leq k \|\mathbb E_{\M \bar\otimes \A} \circ \mathbb E_{\M \bar \otimes \Q}(\xi_n)-\mathbb E_{\M \bar \otimes \Q}(\xi_n)\|_1+ 2\| \xi_n -(a^*\otimes a^*) \xi_n\|_1 + \varepsilon^2. 
\end{align*}

\noindent Taking the $\lim_n$ above and using conditions (ii) for $\xi_n$ and Claim \ref{c1}, we get the desired conclusion. $\hfill\blacksquare$

\vspace{2mm}

\noindent Finally, Claim \ref{c2} and \cite[Theorem 3.2]{CD20} yield the desired conclusion.
\end{proof}

\noindent\emph{Proof of Corollary \ref{main: cor1}}. This follows directly from Corollary \ref{cor: dital implies weak compact} and  Theorem \ref{thm: weak compact implies isp}. \qed 

\section{A counterexample}

In this section we show that our main result, Theorem \ref{mainthm:1}, is optimal in a natural way---namely, it does not hold beyond the weakly compact case. Specifically, we construct a pmp, free, weak mixing action $G \curvearrowright X$ of the non-inner amenable, icc group
$G=\Z_3\wr\F_2$ for which the inclusion $\mathcal L(G)\subset L^\infty(X)\rtimes G$ admits an intermediate von Neumann algebra that does not come from a
measurable factor of the action.

\begin{prop}\label{prop:main}
Consider the groups  $H=\Z_3$, $ \Gamma=\F_2$, $K=H^{(\Gamma)}$, and $G=K\rtimes\Gamma=H\wr\Gamma$. Let $X:=H^\Gamma$ with product uniform measure and consider the following action $G \curvearrowright X$. For every $ag\in G$ ($a\in K$, $g\in \Gamma$), $h\in \Gamma$ and $(x_h)_h\in X$ we let  
\begin{equation}\label{eq:action}
 ((ag)\cdot x)_h=a_h+x_{g^{-1}h}.
\end{equation}
Then $G$ is icc, non-inner amenable and $G \curvearrowright X$ is a pmp, free, weak mixing action.
There exists an intermediate $\mathrm{II}_1$ factor $\mathcal N$ such that $\mathcal L (G)\subsetneq \mathcal N\subsetneq L^\infty(X)\rtimes G$ and $\mathcal N\cap L^\infty(X)=\C1$.
In particular, there is no $G$-invariant von Neumann subalgebra
$\mathcal C\subset L^\infty(X)$ such that $\mathcal N=C\rtimes G$ inside the given
crossed product.
\end{prop}

\noindent Before we proceed to the proof we need the following finite dimensional preliminary elementary result.  This is a finite-dimensional instance
of the construction of Amrutam-Glasner-Glasner
\cite[Propositions 5.2 and 5.5]{AGG24b}. More precisely, the algebra
$\mathcal Q$ below is obtained from their construction by taking
the augmentation ideal $I=(1-p)\mathcal B$. For the reader's
convenience, we include a direct proof. All matrix algebras
below are endowed with their normalized traces.

\begin{lemma}\label{lem:finite}
Let $H=\Z_3$ act on itself by translations, put $\mathcal D=\ell^\infty(H)$,
and identify 
$\mathcal F=\mathcal D\rtimes H=\mathbb M_3(\C)$. Let $u$ implement translation by $1\in H$ and consider  $\mathcal B:=W^*(u)=\mathcal L(H)$. Also consider projection  $p=\frac 13(1+u+u^2)$ and its relative commutant 
 $\mathcal Q=\{p\}'\cap \mathcal F$.
Then we have that 
 $\mathcal B\subsetneq \mathcal Q\subsetneq \mathcal F,
 \mathcal Q\cong\C\oplus \mathbb M_2(\C),\text{ and } \mathcal Q\cap \mathcal D=\C1$.
\end{lemma}

\begin{proof}
In the standard representation on $\ell^2(H)$, $\mathcal D$ is the diagonal algebra,
$u$ is the cyclic permutation matrix, and
\[
 p=\frac13\begin{pmatrix}1&1&1\\1&1&1\\1&1&1\end{pmatrix}
\]
is the rank-one projection onto the constant vectors. Thus
\[
 \mathcal Q=p\mathcal Fp\oplus(1-p)\mathcal F(1-p)\cong\C\oplus \mathbb M_2(\C).
\]
Since $\mathcal B$ is abelian and contains $p$, we have $\mathcal B\subset \mathcal Q$.
As $\dim_\C(\mathcal B)=3,\dim_\C(\mathcal Q)=5$, and $\dim_\C (\mathcal F)=9$ we get $\mathcal B\subsetneq \mathcal Q\subsetneq \mathcal F$.
Finally, whenever $d=\operatorname{diag}(d_0,d_1,d_2)$ commutes with $p$, basic computations show that $d_i=d_j$ for all $i,j=1,3$. Hence $\mathcal D\cap \mathcal Q=\C1$.
\end{proof}

\noindent\emph{Proof of Proposition \ref{prop:main}}. Checking that $G$ is icc and non-inner amenable and the actions $G\curvearrowright X$ and $K \curvearrowright X$ are pmp, free and ergodic are straightforward exercises that we leave to the reader. Moreover since the restriction $\Gamma \curvearrowright X$ is the Bernoulli action then $G\curvearrowright X$ is weak mixing.

Let $\mathcal A=L^\infty(X)$ and $\mathcal R=\mathcal A\rtimes K$. The coordinatewise action and
Lemma~\ref{lem:finite} give the canonical trace-preserving identifications
\begin{equation}\label{eq:tensors}
 \begin{aligned}
 \mathcal R&=\bar\otimes _{h\in\Gamma}(\mathcal F,\tr),&
 \mathcal A&=\bar \otimes_{h\in\Gamma}(\mathcal D,\tr|_{\mathcal D}),&
 \mathcal L(K)&=\bar\otimes _{h\in\Gamma}(\mathcal B,\tr|_{\mathcal B}).
 \end{aligned}
\end{equation}
These identifications hold first on finite sets of coordinates and
then extend to the von Neumann algebras generated by their union.
Conjugation by $g\in\Gamma$ permutes the tensor coordinates by left
translation. Therefore the von Neumann subalgebra $\mathcal P=\bar\otimes_{h\in\Gamma}(\mathcal Q,\tr|_\mathcal Q)\subset \mathcal R$
is $\Gamma$-invariant.
\vskip 0.05in
\noindent Let $v_g$ be the canonical unitary corresponding
to $g\in \Gamma$ and define
\begin{equation}\label{eq:N}
 \mathcal M=\mathcal A\rtimes G=\mathcal R\rtimes\Gamma,\qquad
 \mathcal N=\mathcal P\rtimes\Gamma=W^*(\mathcal P,(v_g)_{g\in\Gamma})\subset \mathcal M.
\end{equation}
Since $\mathcal L(K)\subset \mathcal P$, we have $\mathcal L(G)=\mathcal L(K)\rtimes\Gamma\subset \mathcal N$. Next we show that
\begin{equation}\label{eq:intersections}
 \mathcal N\cap \mathcal R=\mathcal P,\qquad \mathcal L(G)\cap \mathcal R=\mathcal L(K).
\end{equation}
From definitions we clearly have $\mathcal P\subseteq \mathcal N\cap \mathcal R$. For the reverse inclusion, fix $x\in \mathcal N \cap \mathcal R$. As $x\in \mathcal N=\mathcal P\rtimes \Gamma$, using its canonical Fourier expansion one has $x=\sum_{g\in \Gamma} x_g v_g$ for $x_g=\mathbb E_{\mathcal P} (xv_g^{-1})\in \mathcal P$. If $\mathbb E_{\mathcal R} \,:\,\mathcal M \to \mathcal R$ is the canonical, trace-preserving conditional expectation, as $\mathcal P\subseteq \mathcal R$, we have that  $x= \mathbb E_{\mathcal R}(x)= \sum_g \mathbb E_{\mathcal R}( x_g v_g )= \sum_g x_g \mathbb E_{\mathcal R}(v_g)=x_e\in \mathcal P$.  The second assertion can be proved in a similar way.

\noindent Choose an element $q\in\mathcal Q\setminus \mathcal B$ in the tensor coordinate indexed
by $e\in\Gamma$. The expectation onto that coordinate maps $\mathcal L(K)$
into $\mathcal B$, so $q\notin\mathcal L(K)$. By \eqref{eq:intersections},
$q\notin\mathcal L(G)$, whereas $q\in \mathcal P\subset \mathcal N$. Thus
$\mathcal L(G)\subsetneq \mathcal N$.

\vskip 0.05in
\noindent For $h\in\Gamma$, let $u_h$ denote the canonical unitary corresponding
to $\delta_h\in K<G$, and consider the element $p_h:=\frac13(1+u_h+u_h^2)$. From definitions it follows that $p_h$ commutes with every element of $\mathcal P$.
\vskip 0.05in
\noindent Let $f\in \mathcal P\cap \mathcal A$ and write $\beta_h={\rm Ad}(u_h)|_{\mathcal A}$. Since $[f,p_h]=0$,
\[
 0=3[f,p_h]
   =(f-\beta_h(f))u_h+(f-\beta_h^2(f))u_h^2.
\]
Since the elements $\delta_h,2\delta_h \in nK$ are distinct, uniqueness
of Fourier coefficients in $\mathcal A\rtimes K$ implies that $\beta_h(f)=f$.
This holds for all $h\in\Gamma$ and hence  $f$ is $K$-invariant. By the
ergodicity of $K\curvearrowright X$, $f$ is scalar.
Using \eqref{eq:intersections}, we obtain
\begin{equation}\label{eq:scalar}
 \mathcal N\cap \mathcal A=\mathcal P\cap \mathcal A=\C1.
\end{equation}
Since $A\ne\C1$, note that \eqref{eq:scalar} also proves $\mathcal N\subsetneq \mathcal M$. 
\vskip 0.05in
\noindent To this end, assume, by contradiction,  that $\mathcal N=\mathcal C\rtimes G$ for a $G$-invariant von Neumann subalgebra
$\mathcal C\subset \mathcal A$. Taking identity Fourier coefficients shows that
$(\mathcal C\rtimes G)\cap \mathcal A=\mathcal C$. Hence \eqref{eq:scalar} implies $C=\C1$,
and thus $\mathcal N=\mathcal L(G)$, contradicting the strict inclusion. \qed

\bigskip
\noindent
Ionu\c{t} Chifan\\
Email: \texttt{ionut-chifan@uiowa.edu}\\
14 MacLean Hall, Department of Mathematics, The University of Iowa, Iowa City, 52246, IA, USA

\medskip
\noindent
Zachary Flood\\
Email: \texttt{zachary-flood@uiowa.edu}\\
14 MacLean Hall, Department of Mathematics, The University of Iowa, Iowa City, 52246, IA, USA

\end{document}